\documentclass[12pt]{amsart}
\usepackage{a4wide,enumerate,color}
\allowdisplaybreaks

\let\pa\partial  
\let\na\nabla  
\let\eps\varepsilon  
\newcommand{\diver}{\operatorname{div}}
\newcommand{\N}{\mathbb{N}}
\newcommand{\R}{\mathbb{R}}
\newcommand{\E}{\mathbb{E}}
\newcommand{\F}{{\mathcal F}}

\newtheorem{theorem}{Theorem}   
\newtheorem{lemma}[theorem]{Lemma}   
\newtheorem{proposition}[theorem]{Proposition}   
\newtheorem{remark}[theorem]{Remark}

\begin{document}  

\title[Rigorous mean-field limit and cross diffusion]{Rigorous mean-field limit 
and cross diffusion}

\author[L. Chen]{Li Chen}
\address{University of Mannheim, Department of Mathematics, 68131 Mannheim, Germany}
\email{chen@math.uni-mannheim.de}

\author[E. S. Daus]{Esther S. Daus}
\address{Institute for Analysis and Scientific Computing, Vienna University of  
	Technology, Wiedner Hauptstra\ss e 8--10, 1040 Wien, Austria}
\email{esther.daus@tuwien.ac.at} 

\author[A. J\"ungel]{Ansgar J\"ungel}
\address{Institute for Analysis and Scientific Computing, Vienna University of  
	Technology, Wiedner Hauptstra\ss e 8--10, 1040 Wien, Austria}
\email{juengel@tuwien.ac.at} 

\date{\today}

\thanks{The first author acknowledges support from the DFG project CH955/3-1. 
The last two authors acknowledge partial support from   
the Austrian Science Fund (FWF), grants F65, P27352, P30000, and W1245.
The authors thank Dr.\ Nicola Zamponi for his help for the proof of Lemma \ref{lem.sqrt}
in the appendix} 

\begin{abstract}
The mean-field limit in a weakly interacting stochastic many-particle system 
for multiple population species in the whole space is proved. 
The limiting system consists of cross-diffusion equations,
modeling the segregation of populations. 
The mean-field limit is performed in two steps: First, the many-particle
system leads in the large population limit to an intermediate nonlocal 
diffusion system. The local cross-diffusion system is then obtained from
the nonlocal system when the interaction potentials approach the Dirac delta
distribution. The global existence of the limiting and the intermediate 
diffusion systems is shown for small initial data, and an error estimate is given.
\end{abstract}



\maketitle


\section{Introduction}\label{sec.intro}

Cross-diffusion models are systems of quasilinear 
parabolic equations with a nondiagonal diffusion matrix. 
They arise in many applications in cell biology, 
multicomponent gas dynamics, population dynamics, etc.\ \cite{Jue16}.
To understand the range of validity of these diffusion systems, it is important to 
derive them from first principles or from more general models. In the literature,
cross-diffusion systems were derived from random walks on lattice models
\cite{ZaJu17}, the kinetic Boltzmann equation \cite{BGS15},
reaction-diffusion systems \cite{CoDe14,IzMi08}, 
or from stochastic many-particle systems \cite{Ste00}.
We derive in this paper rigorously the $n$-species cross-diffusion system
\begin{equation}\label{1.cds}
\begin{aligned}
  \pa_t u_i - \sigma_i\Delta u_i &= \diver\bigg(\sum_{j=1}^n a_{ij}u_i\na u_j\bigg)
	\quad\mbox{in }\R^d,\ t>0, \\
  u_i(0) &= u_i^0, \quad i=1,\ldots,n,
\end{aligned}
\end{equation}
where $\sigma_i>0$ and $a_{ij}$ are real numbers, starting from a stochastic
many-particle system for multiple species. System \eqref{1.cds} describes the 
diffusive dynamics of populations subject to segregation effects modeled by
the term on the right-hand side \cite{GaSe14}.

\subsection{Setting of the problem}

We consider $n$ subpopulations of interacting individuals moving in the 
whole space $\R^d$ with the particle numbers $N_i\in\N$, $i=1,\ldots,n$. 
We take $N_i=N$ to simplify the notation. The individuals are represented by
the stochastic processes $X_{\eta,i}^{k,N}(t)$ evolving according to
\begin{equation}\label{1.Xi}
\begin{aligned}
  dX_{\eta,i}^{k,N}(t) &= -\sum_{j=1}^n\frac{1}{N}\sum_{\ell=1}^N
	\na V_{ij}^\eta\big(X_{\eta,i}^{k,N}(t)-X_{\eta,j}^{\ell,N}(t)\big)dt
	+ \sqrt{2\sigma_i}dW_i^k(t), \\
  X_{\eta,i}^{k,N}(0) &= \xi_i^k, \quad i=1,\ldots,n,\ k=1,\ldots,N,
\end{aligned}
\end{equation}
where $(W_i^k(t))_{t\ge 0}$ are $d$-dimensional Brownian motions,
the initial data $\xi_i^1,\ldots,\xi_i^N$ are independent and 
identically distributed random variables with the common probability density function
$u_i^0$, and the interaction potential $V_{ij}^\eta$ is given by
$$
  V_{ij}^{\eta}(x) = \frac{1}{\eta^d}V_{ij}\bigg(\frac{|x|}{\eta}\bigg), 
	\quad x\in\R^d,\ i,j=1,\ldots,n.
$$
Here, $V_{ij}$ is a given smooth function and $\eta>0$ the scaling parameter.
The scaling is chosen in such a way that the $L^1$ norm of $V_{ij}^\eta$ stays
invariant and $V_{ij}^\eta\to a_{ij}\delta$ in the sense of distributions
as $\eta\to 0$, where $\delta$ denotes the Dirac delta distribution.

The mean-field limit $N\to\infty$, $\eta\to 0$ 
has to be understood in the following sense.
For fixed $\eta>0$, the many-particle model \eqref{1.Xi} is approximated 
for $N\to\infty$ by the intermediate stochastic system
\begin{equation}\label{1.barXi}
\begin{aligned}
  d\bar X_{\eta,i}^k(t) &= -\sum_{j=1}^n(\na V_{ij}^\eta* u_{\eta,j})
	(\bar X_{\eta,i}^k(t),t)dt + \sqrt{2\sigma_i} dW_i^k(t), \\
	\bar X_{\eta,i}^k(0) &= \xi_i^k, \quad i=1,\ldots,n,\ k=1,\ldots,N,
\end{aligned}
\end{equation}
where $u_{j,\eta}$ is the probability density function of $\bar X_{j,\eta}^k$,
satisfying the nonlocal diffusion system
\begin{equation}\label{1.uieta}
\begin{aligned}
  \pa_t u_{\eta,i} &= \sigma_i\Delta u_{\eta,i} + \diver\bigg(\sum_{j=1}^n u_{\eta,i}
	\na V_{ij}^\eta * u_{\eta,j}\bigg)\quad\mbox{in }\R^d,\ t>0, \\
	u_{\eta,i}(0) &= u_{i}^0, \quad i=1,\ldots,n.
\end{aligned}
\end{equation}
Observe that the intermediate system depends on $k$ only through the initial data.
Then, passing to the limit $\eta\to 0$ in the intermediate system, 
the limit $\na V_{ij}^\eta * u_{\eta,j} \to a_{ij}\na u_j$ in $L^2$ leads to the
limiting stochastic system
\begin{equation}\label{1.hatXi}
\begin{aligned}
  d\widehat X_i^k(t) &= -\sum_{j=1}^n a_{ij}\na u_j(\widehat X_i^k(t),t)dt
	+ \sqrt{2\sigma_i}dW_i^k(t), \\
	\widehat X_i^k(0) &= \xi_i^k, \quad i=1,\ldots,n,\ k=1,\ldots,N,
\end{aligned}
\end{equation}
and the law of $\widehat X_i^k$ is a solution to the limiting cross-diffusion 
system \eqref{1.cds}. The main result of this paper
is the proof of the error estimate
$$
  \E\bigg(\sum_{i=1}^n\sup_{0<s<t}\sup_{k=1,\ldots,N}\big|X_{\eta,i}^{k,N}(s)
	- \widehat X_i^k(s)\big|\bigg) \le C(t)\eta,
$$
if we choose $\eta$ and $N$ such that $\eta^{-(2d+4)}\le \eps\log N$ holds
and $\eps>0$ can be any small number.

\subsection{State of the art}

Mean-field limits were investigated intensively in the last decades to derive,
for instance, reaction-diffusion equations \cite{DFL86} or McKean-Vlasov equations 
\cite{DaDe96,Gae88} (also see the reviews \cite{Gol03,JaWa17}). 
Oelschl\"ager \cite{Oel84} considered in the 1980s a weakly interacting 
particle system of $N$ particles and proved that in the limit $N\to\infty$, 
the stochastic system converges to a deterministic nonlinear process.
Later, he generalized his approach to systems of reaction-diffusion equations
\cite{Oel89}.

The analysis of quasilinear diffusion systems started more recently. 
The chemotaxis system was derived by Stevens \cite{Ste00} from a stochastic 
many-particle system with a limiting procedure that is based on Oelschl\"ager's
work. Reaction-diffusion systems 
with nonlocal terms were derived in \cite{KaSu13} as the mean-field limit of
a master equation for a vanishing reaction radius; also see \cite{IRS12}.
The two-species Maxwell-Stefan equations were found to be the hydrodynamic
limit system of two-component Brownian motions with singular interactions \cite{Seo18}.
Nonlocal Lotka-Volterra systems with cross diffusion were obtained in the
large population limit of point measure-valued Markov processes by Fontbona and
M\'el\'eard \cite{FoMe15}. Moussa \cite{Mou17} then proved the limit from the 
nonlocal to the local diffusion system (but only for
triangular diffusion matrices), which gives the Shigesada-Kawasaki-Teramoto
cross-diffusion system. A derivation of a space discretized
version of this system from a Markov chain model was presented in \cite{DDD18}.
Another nonlocal mean-field model was analyzed in \cite{BPRSW18}.

Our system \eqref{1.cds} is different from the aforementioned
Shigesada-Kawasaki-Teramoto system
$$
  \pa_t u_i = \Delta\bigg(u_i\sum_{j=1}^n a_{ij}u_j\bigg)
	= \diver\bigg(\sum_{j=1}^n a_{ij}u_i\na u_j\bigg)
	+ \diver\bigg(\sum_{j=1}^n a_{ij}u_j\na u_i\bigg)
$$
derived in \cite{FoMe15,Mou17}.
Our derivation produces the first term on the right-hand side.
The reason for the difference is that in \cite{FoMe15}, 
the diffusion coefficient $\sigma_i$ in 
\eqref{1.Xi} is assumed to depend on the convolutions $W_{ij} * u_j$ for some
functions $W_{ij}$ -- yielding the last term in the previous equation --, 
while we have assumed a constant
diffusion coefficient. It is still an open problem to derive the general
Shigesada-Kawasaki-Teramoto system; the approach of Moussa \cite{Mou17} requires
that $a_{ij}=0$ for $j<i$.

System \eqref{1.cds} was also investigated in the literature.
A formal derivation from the intermediate diffusion system \eqref{1.uieta}
was performed by Galiano and Selgas \cite{GaSe14}, while probabilistic representations 
of \eqref{1.cds} were presented in \cite{Bel16}. A rigorous derivation from
the stochastic many-particle system \eqref{1.Xi} is still missing in the literature.
In this paper, we fill this gap by extending the technique of
\cite{CGK18} to diffusion systems. Compared to \cite{CGK18}, the argument to
derive the uniform estimates is more involved and involves a nonlinear
Gronwall argument (see Lemma \ref{lem.sqrt} in the appendix).

The global existence of solutions to \eqref{1.cds} for general initial data 
and coefficients $a_{ij}\ge 0$ is an open problem. 
The reason is that we do not know any entropy structure of \eqref{1.cds}.
For the two-species system, 
Galiano and Selgas \cite{GaSe14} proved the global existence of weak solutions 
in a bounded domain with no-flux boundary conditions 
under the condition $4a_{11}a_{22}>(a_{12}+a_{21})^2$. 
The idea of the proof is to show that $H(u)=\int u_i(\log u_i-1) dx$ is a 
Lyapunov functional (entropy). The condition can be weakened to
$a_{11}a_{22}>a_{12}a_{21}$ using the modified entropy 
$H_1(u)=\int(a_{21}u_1(\log u_1-1)+a_{12}u_2(\log u_2-1))dx$, but this
is still a weak cross-diffusion condition.

We use the following notation throughout the paper. 
We write $\|\cdot\|_{L^p}$ and $\|\cdot\|_{H^s}$ for the norms
of $L^p=L^p(\R^d)$ and $H^s=H^s(\R^d)$, respectively. Furthermore,
$|u|^2 = \sum_{i=1}^n u_i^2$ for $u=(u_1,\ldots,u_n)\in\R^n$ and
$\|u\|_{L^p}^2=\sum_{i=1}^n\|u_i\|_{L^p}^2$ for functions $u=(u_1,\ldots,u_n)$.
We use the notation $u(t)=u(\cdot,t)$ for functions depending on $x$ and $t$,
and $C>0$ is a generic constant whose value may change from line to line.

\subsection{Main results}

The first two results are concerned with the solvability of the
nonlocal diffusion system \eqref{1.uieta} and the limiting cross-diffusion
system \eqref{1.cds}. The existence results are needed for our main result,
Theorem \ref{thm.main} below.

We impose the following assumptions on the interaction potential.
Let $V_{ij}\in C_0^2(\R^d)$ be such that $\operatorname{supp}(V_{ij})
\subset B_1(0)$ for $i,j=1,\ldots,n$. Then $V_{ij}^\eta(x)=\eta^{-d}V_{ij}(|x|/\eta)$
for $\eta>0$ satisfies  $\operatorname{supp}(V_{ij}^\eta)\subset B_\eta(0)$ and
$$
  a_{ij} := \int_{\R^d} V_{ij}(|x|)dx = \int_{\R^d}V_{ij}^\eta(x)dx,
	\quad i,j=1,\ldots,n.
$$
As the potential may be negative (and $a_{ij}$ may be negative too), we introduce
$$
  A_{ij} := \|V_{ij}\|_{L^1} = \|V_{ij}^\eta\|_{L^1}, \quad i,j=1,\ldots,n.
$$

\begin{proposition}[Existence for the nonlocal diffusion system]\label{prop.ex}
\ Let $u^0=(u_1^0,\ldots,u_n^0)\in$ $H^s(\R^d;\R^n)$ 
with $s>d/2+1$ and $u_i^0\ge 0$ in $\R^d$ and assume that
\begin{equation}\label{1.small}
  \|u^0\|_{H^s} \le \frac{\sigma}{C^*\sum_{i,j=1}^n A_{ij}},
\end{equation}
where $\sigma=\min_{i=1,\ldots,n}\sigma_i>0$ and 
$C^*>0$ is a constant only depending on $s$ and $d$. 
Then there exists a global
solution $u_\eta=(u_{\eta,1},\ldots,u_{\eta,n})$ to problem \eqref{1.uieta} such that
$u_{\eta,i}(t)\ge 0$ in $\R^d$, $t>0$,
$u_\eta\in L^\infty(0,\infty;H^s(\R^d;\R^n)\cap L^2(0,\infty;H^{s+1}(\R^d;\R^n))$, and
$$
	\sup_{t>0}\|u_\eta(t)\|_{H^s} \le \|u^0\|_{H^s}.
$$
Moreover, if for some $0<\gamma<\sigma$ the slightly stronger condition
\begin{equation}\label{1.small2}
  \|u^0\|_{H^s} \le \frac{\sigma-\gamma}{C^*\sum_{i,j=1}^n A_{ij}}
\end{equation}
holds, then the solution is unique and
\begin{equation}\label{1.est}
  \sup_{t>0}\|u_\eta(t)\|_{H^s}^2 + \gamma\|\na u_\eta\|_{L^2(0,\infty;H^{s+1})}^2 
	\le \|u^0\|_{H^s}^2.
\end{equation}
\end{proposition}

Since we do not use the structure of the equations, we can only expect the
global existence of solutions for sufficiently small initial data.
The proof of this result is based on the Banach fixed-point theorem and
a priori estimates and is rather standard. We present it for completeness.

\begin{proposition}[Existence for the limiting cross-diffusion system]
\label{prop.ex2}
Let $u^0\in H^s(\R^d;\R^n)$ with $s>d/2+1$ such that $u_i^0\ge 0$
in $\R^d$ and \eqref{1.small2} holds. Then there exists a unique glob\-al solution 
$u=(u_1,\ldots,u_n)$ to problem \eqref{1.cds} satisfying $u_i(t)\ge 0$ 
in $\R^d$, $t>0$, $u\in L^\infty(0,\infty;H^s(\R^d;\R^n))\cap
L^2(0,\infty;H^{s+1}(\R^d;\R^n))$, and
\begin{equation}\label{1.estu}
  \sup_{t>0}\|u(t)\|_{H^s}^2 + \gamma\|u\|_{L^2(0,\infty;H^{s+1})}^2
	\le \|u^0\|_{H^s}^2.
\end{equation}
Moreover, let $u_\eta$ be the solution to problem \eqref{1.uieta}. Then the following
error estimate holds for any $T>0$:
\begin{equation}\label{1.error}
  \|u_\eta-u\|_{L^\infty(0,T;L^2)} + \|\na(u_\eta-u)\|_{L^2(0,T;L^2)} \le C(T)\eta
\end{equation}
for some constant $C(T)>0$. 
\end{proposition}

The proposition is proved by performing the limit $\eta\to 0$ in
\eqref{1.uieta} which is possible in view of the uniform estimate \eqref{1.est}.
The error estimate \eqref{1.error} follows from the uniform bounds and the
smallness condition \eqref{1.small}. 

For our main result, we need to make precise the stochastic setting. 
Let $(\Omega,\F,(\F_t)_{t\ge 0},{\mathcal P})$ be a filtered probability space
and let $(W_i^k(t))_{t\ge 0}$ for $i=1,\ldots,n$, $k=1,\ldots,N$ be 
$d$-dimensional $\F_t$-Brownian motions that are independent of the 
random variables $\xi_i^k$. We assume that the Brownian motions are independent
and that the initial data $\xi_i^1,\ldots,\xi_i^N$ are independent and 
identically distributed random variables with the common probability density function
$u_i^0$. 

We prove in Section \ref{sec.stoch} that if $s>d/2+2$ and the initial density
$u^0$ satisfies the smallness condition \eqref{1.small}, the stochastic 
differential systems \eqref{1.Xi}, \eqref{1.barXi}, and \eqref{1.hatXi}
have pathwise unique strong solutions; also see Remark \ref{rem}.

\begin{theorem}[Error estimate for the stochastic system]\label{thm.main}
Under the aforementioned as\-sump\-tions, let $s>d/2+2$ and let
$X_{\eta,i}^{k,N}$ and $\widehat X_{i}^k$ be solutions to the problems
\eqref{1.Xi} and \eqref{1.hatXi}, respectively. Furthermore, let $0<\eps<1$
be sufficiently small and choose $N\in\N$ such that $\eps\log N\ge \eta^{-2d-4}$.
Then, for any $t>0$, 
$$
  \E\bigg(\sum_{i=1}^n\sup_{0<s<t}\sup_{k=1,\ldots,N}
	\big|X_{\eta,i}^{k,N}(s) - \widehat X_{i}^k(s)\big|\bigg) 
	\le C(t)\eta,
$$
where the constant $C(t)$ depends on $t$, $n$, $\|D^2 V_{ij}\|_{L^\infty}$,
and the initial datum $u^0$.
\end{theorem}

The idea of the proof is to derive error estimates for the differences
$X_{\eta,i}^{k,N}-\bar X_{\eta,i}^k$ and $\bar X_{\eta,i}^k-\widehat X_i^k$
(where $\bar X_{\eta,i}^{k,N}$ solves \eqref{1.barXi}) and to use
\begin{align*}
  \E\bigg(\sum_{i=1}^n\sup_{0<s<t}\sup_{k=1,\ldots,N}
	\big|X_{\eta,i}^{k,N} - \widehat X_i^k\big|\bigg)
	&\le \E\bigg(\sum_{i=1}^n\sup_{0<s<t}\sup_{k=1,\ldots,N}
	\big|X_{\eta,i}^{k,N}-\bar X_{\eta,i}^k\big|\bigg) \\
	&\phantom{xx}{}+ \E\bigg(\sum_{i=1}^n\sup_{0<s<t}\sup_{k=1,\ldots,N}
	\big|\bar X_{\eta,i}^k-\widehat X_i^k\big|\bigg).
\end{align*}
The expectations on the right-hand side are estimated by taking the difference
of the solutions to the corresponding stochastic differential equations,
exploiting the Lipschitz continuity of $\na V_{ij}^\eta$, and observing that
$\|V_{ij}^\eta * \na u_j - a_{ij}\na u_j\|_{L^2(0,t;L^2)}\le C\eta$.

The paper is organized as follows. Sections \ref{sec.ex} and \ref{sec.ex2} are 
concerned with the proof of Propositions \ref{prop.ex} and \ref{prop.ex2},
respectively. The existence of solutions to the stochastic systems is shown
in Section \ref{sec.stoch}. The main result (Theorem \ref{thm.main}) is then
proved in Section \ref{sec.proof}. Finally, the appendix recalls some auxiliary
results needed in our analysis.


\section{Existence for the nonlocal diffusion system \eqref{1.uieta}}\label{sec.ex}

We show Proposition \ref{prop.ex} whose proof is split into several lemmas.

\begin{lemma}[Local existence of solutions]\label{lem.loc}
Let $u^0=(u_1^0,\ldots,u_n^0)\in H^s(\R^d;\R^n)$ with $s>d/2+1$ and
$u_i^0\ge 0$ in $\R^d$. Then there
exists $T^*>0$ such that \eqref{1.uieta} possesses the unique solution
$u_\eta=(u_{\eta,1},\ldots,u_{\eta,n})\in L^\infty(0,T^*;H^s(\R^d;\R^n))\cap
L^2(0,T^*;H^{s+1}(\R^d;\R^n))$ satisfying $u_{\eta,i}(t)\ge 0$ in $\R^d$ for $t>0$.
The time $T^*>0$ depends on $\|u^0\|_{H^s}$ such that $T^*\to 0$ if
$\|u^0\|_{H^s}\to\infty$.
\end{lemma}

\begin{proof}
The idea is to apply the Banach fixed-point theorem. For this, we introduce
\begin{align}
  Y &= \Big\{v\in L^\infty(0,T^*;H^s(\R^d;\R^n))\cap L^2(0,T^*;H^{s+1}(\R^d;\R^n)): \\
	&\phantom{xxx}\sup_{0<t<T^*}\|v(\cdot,t)\|_{H^s}^2 \le M:= 1 + \|u^0\|_{H^s}^2\Big\},
\end{align}
endowed with the metric 
$\operatorname{dist}(u,w)=\sup_{0<t<T^*}\|(u-w)(t)\|_{L^2}$, where $T^*>0$ will 
be determined later. The fixed-point operator $S:Y\to Y$ is defined by 
$Sv=u$, where $u$ is the unique solution to the Cauchy problem
\begin{equation}\label{2.cp}
  \pa_t u_i = \sigma_i\Delta u_i + \diver\bigg(\sum_{j=1}^n u_i^+\na V_{ij}^\eta
	* v_j\bigg), \quad u_i(0)=u_i^0 \mbox{ in }\R^d,
\end{equation}
and $u_i^+=\max\{0,u_i\}$. 
The existence of a unique solution $u\in C^0([0,T];H^s(\R^d))\cap L^2(0,T;H^{s+1}(\R^d))$
to this linear advection-diffusion problem follows from semigroup theory since
$u^0\in H^s(\R^d;\R^n)$. Taking the test function $u_i^- = \min\{0,u_i\}$ in the
weak formulation of \eqref{2.cp} yields
$$
  \frac12\frac{d}{dt}\int_{\R^d}(u_i^-)^2 dx
	+ \sigma_i\int_{\R^d}|\na u_i^-|^2 dx
	= -\int_{\R^d}\sum_{j=1}^n u_i^+(\na V_{ij}^\eta * v_j)\cdot\na u_i^- dx.
$$
Since $u_i^+\na u_i^-=0$ in $\R^d$, we infer that $u_i^-=0$ in $\R^d$, showing
that $u_i(t)$ is nonnegative for all $t\in(0,T^*)$.

We prove that $\sup_{0<t<T^*}\|u(\cdot,t)\|_{H^s}^2 \le M$ for sufficiently small
values of $T^*>0$. Then $u\in Y$
and $S:Y\to Y$ is well defined. We apply the differential operator $D^\alpha$
for an arbitrary multi-index $\alpha\in\N^d$ of order $|\alpha|\le s$ to \eqref{2.cp},
multiply the resulting equation by $D^\alpha u_{i}$, and integrate
over $\R^d$:
$$
  \frac12\frac{d}{dt}\int_{\R^d}|D^\alpha u_i|^2 dx
	+ \sigma_i\int_{\R^d}|\na D^\alpha u_i|^2 dx
	= -\int_{\R^d}\sum_{j=1}^n D^\alpha(u_i\na V_{ij}^\eta * v_j)\cdot\na D^\alpha u_i dx.
$$
We sum these equations from $i=1,\ldots,n$, apply the Cauchy-Schwarz inequality
to the integral on the right-hand side, and the Moser-type calculus inequality
(Lemma \ref{lem.moser}):
\begin{align*}
  \frac12&\frac{d}{dt}\int_{\R^d}|D^\alpha u|^2 dx
	+ \sigma\int_{\R^d}|\na D^\alpha u|^2 dx
	\le \sum_{i,j=1}^n \big\|D^\alpha(u_i\na V_{ij}^\eta * v_j)\big\|_{L^2}
	\|\na D^\alpha u_i\|_{L^2} \\
	&\le C(\eps)\sum_{i,j=1}^n \big\|D^\alpha(u_i\na V_{ij}^\eta * v_j)\big\|_{L^2}^2
	+ \eps n\|\na D^\alpha u\|_{L^2}^2 \\
	&\le C(\eps)\sum_{i,j=1}^n\Big(\|u_i\|_{L^\infty}
	\big\|D^s(\na V_{ij}^\eta * v_j)\big\|_{L^2} + \|D^s u_i\|_{L^2}
	\|\na V_{ij}^\eta * v_j\|_{L^\infty}\Big)^2 \\
	&\phantom{xx}{}+ \eps n\|\na D^\alpha u\|_{L^2}^2,
\end{align*}
where we recall that $\sigma=\min_{i=1,\ldots,n}\sigma_i>0$
and $\eps$ is any positive number. 
The last term on the right-hand side can be absorbed by the second term on the
left-hand side if $\eps\le\sigma/(2n)$. Hence, summing over all multi-indices
$\alpha$ of order $|\alpha|\le s$, using Young's convolution inequality
(Lemma \ref{lem.young}), and the inequality $\|\na V_{ij}^\eta\|_{L^1}\le C(\eta)$, 
we find that
\begin{align*}
  \frac{d}{dt}\|u\|_{H^s}^2 + \frac{\sigma}{2}\|\na u\|_{H^s}^2
	&\le C\sum_{i,j=1}^n\Big(\|u\|_{L^\infty}\|\na V_{ij}^\eta\|_{L^1}\|D^s v_j\|_{L^2}
	+ \|D^s u_i\|_{L^2}\|\na V_{ij}^\eta\|_{L^1}\|v_j\|_{L^\infty}\Big)^2 \\
  &\le C(\eta)\|u\|_{H^s}^2\|v\|_{H^s}^2,
\end{align*}
where in the last step we have taken into account the continuous embedding
$H^s(\R^d)\hookrightarrow L^\infty(\R^d)$.
As $v\in Y$ and consequently $\|v(t)\|_{H^s}^2\le M$, we infer that
$$
  \frac{d}{dt}\|u\|_{H^s}^2 \le C(\eta) M\|u\|_{H^s}^2.
$$
Gronwall's inequality then yields
$$
  \|u(t)\|_{H^s}^2 \le \|u^0\|_{H^s}^2 e^{C(\eta)Mt}
	= (M-1)e^{C(\eta)MT^*} \le M, \quad 0<t<T^*,
$$
if we choose $T^*>0$ so small that $e^{C(\eta)MT^*}\le M/(M-1)$. 
We conclude that $u\in Y$.

Note that the time $T^*$ depends on $M$ and hence on $u^0$ in such a way that
$T^*$ becomes smaller if $\|u^0\|_{H^s}$ is large but $T^*>0$ is bounded from below
if $\|u^0\|_{H^s}$ is small.

It remains to show that the map $S:Y\to Y$ is a contraction, possibly for a smaller
value of $T^*>0$. Let $v,w\in Y$ and take the difference of the equations satisfied
by $Sv$ and $Sw$, respectively:
\begin{align*}
  \pa_t&\big((Sv)_i-(Sw)_i\big)
	- \sigma_i\Delta\big((Sv)_i-(Sw)_i\big) \\
	&= \diver\bigg(\sum_{j=1}^n\big((Sv)_i-(Sw)_i\big)\na V_{ij}^\eta * v_j\bigg)
	+ \diver\bigg(\sum_{j=1}^n (Sw)_i\na V_{ij}^\eta * (v_j-w_j)\bigg).
\end{align*}
Multiplying these equations by $(Sv)_i-(Sw)_i$, summing from $i=1,\ldots,n$,
integrating over $\R^d$, and using the Cauchy-Schwarz inequality leads to
\begin{align*}
  \frac12\frac{d}{dt}&\int_{\R^d}|Sv-Sw|^2 dx 
	+ \sigma\int_{\R^d}|\na(Sv-Sw)|^2 dx \\
	&\le \sum_{i,j=1}^n \big\|((Sv)_i-(Sw)_i)\na V_{ij}^\eta * v_j\big\|_{L^2}
	\|\na((Sv)_i-(Sw)_i)\|_{L^2} \\
	&\phantom{xx}{}+ \sum_{i,j=1}^n \big\|(Sw)_i\na V_{ij}^\eta * (v_j-w_j)\big\|_{L^2}
	\|\na((Sv)_i-(Sw)_i)\|_{L^2}.
\end{align*}
We deduce from Young's convolution inequality that
\begin{align*}
  \frac12\frac{d}{dt}&\int_{\R^d}|Sv-Sw|^2 dx 
	+ \frac{\sigma}{2}\int_{\R^d}|\na(Sv-Sw)|^2 dx \\
	&\le C(\sigma)\sum_{i,j=1}^n\big\|((Sv)_i-(Sw)_i)\na V_{ij}^\eta * v_j\big\|_{L^2}^2
	+ C(\sigma)\sum_{i,j=1}^n \big\|(Sw)_i\na V_{ij}^\eta * (v_j-w_j)\big\|_{L^2}^2 \\
	&\le C(\sigma)\sum_{i,j=1}^n\|\na V_{ij}^\eta * v_j\|_{L^\infty}^2
	\|(Sv)_i-(Sw)_i\|_{L^2}^2 \\
	&\phantom{xx}{}+ C(\sigma) \max_{i=1,\ldots,n}\|(Sw)_i\|_{L^\infty}^2\sum_{i,j=1}^n
	\|\na V_{ij}^\eta * (v_j-w_j)\|_{L^2}^2 \\
	&\le C(\sigma)\sum_{i,j=1}^n\|\na V_{ij}^\eta\|_{L^1}^2\|v_j\|_{L^\infty}^2
	\|Sv-Sw\|_{L^2}^2 \\
	&\phantom{xx}{}+ C(\sigma)\|Sw\|_{L^\infty}^2\sum_{i,j=1}^n 
	\|\na V_{ij}^\eta\|_{L^1}^2\|v_j-w_j\|_{L^2}^2.
\end{align*}
By definition of the metric on $Y$, we have shown that
$$
  \frac{d}{dt}\|Sv-Sw\|_{L^2}^2
	\le C_1(\sigma,\eta,M)\|Sv-Sw\|_{L^2}^2 + C_2(\sigma,\eta,M)\|v-w\|_{L^2}^2.
$$
The constants $C_1$ and $C_2$ depend on $M$ (and hence on $u^0$) in such a way
that they become larger if $\|u^0\|_{H^s}$ is large but they stay bounded for
small values of $\|u^0\|_{H^s}$.
Thus, because of $v(0)=w(0)$, Gronwall's inequality gives
\begin{align*}
  \|Sv(t)-Sw(t)\|_{L^2}^2 
	&\le C_2(\sigma,\eta,M)\int_0^t e^{C_1(\sigma,\eta)(t-s)}\|v(s)-w(s)\|^2ds \\
	&\le C_2(\sigma,\eta,M)\big(e^{C_1(\sigma,\eta)t}-1\big)
	\sup_{0<s<t}\|v(s)-w(s)\|_{L^2}^2,
\end{align*}
and the definition of the metric leads to
$$
  \operatorname{dist}(Sv,Sw)^2 
	\le C_2(\sigma,\eta,M)\big(e^{C_1(\sigma,\eta,M)T^*}-1\big)\operatorname{dist}(v,w)^2.
$$
Then, choosing $T^*>0$ such that 
$C_2(\sigma,\eta,M)(e^{C_1(\sigma,\eta,M)T^*}-1)\le 1/2$ 
shows that $S:Y\to Y$ is a contraction. Again, $T^*$ depends on $u^0$ but
it is bounded from below for small values of $\|u^0\|_{H^s}$.
Thus, we can apply the Banach fixed-point theorem, finishing the proof.
\end{proof}

\begin{lemma}[A priori estimates]
Let assumption \eqref{1.small2} hold.
For the local solution $u_\eta$ to problem \eqref{1.uieta}, the uniform
estimate \eqref{1.est} holds. In particular, the solution $u_\eta$ can be
extended to a global one.
\end{lemma}

\begin{proof}
We proceed similarly as in the proof of Lemma \ref{lem.loc}. We choose $\alpha$
of order $|\alpha|\le s$, apply the operator $D^\alpha$ on both sides of 
\eqref{1.uieta}, multiply the resulting equation by $D^\alpha u_{\eta,i}$,
and integrate over $\R^d$. By the Cauchy-Schwarz inequality, the Moser-type
calculus inequality, and Young's convolution inequality and writing $u_i$ 
instead of $u_{\eta,i}$, we obtain
\begin{align*}
  \frac12&\frac{d}{dt}\int_{\R^d}|D^\alpha u|^2 dx
	+ \sigma\int_{\R^d}|\na D^\alpha u|^2 dx \\
	&\le \sum_{i,j=1}^n\big\|D^\alpha(u_iV_{ij}^\eta * \na u_j)\big\|_{L^2}
	\|\na D^\alpha u_i\|_{L^2} \\
	&\le C_M\sum_{i,j=1}^n\Big(\|u_i\|_{L^\infty}\|D^s(V_{ij}^\eta * \na u_j)\|_{L^2}
	+ \|D^s u_i\|_{L^2}\|V_{ij}^\eta * \na u_j\|_{L^\infty}\Big)
	\|\na D^\alpha u_i\|_{L^2} \\
	&\le C_M\sum_{i,j=1}^n \Big(C\|u\|_{H^s}\|V_{ij}^\eta\|_{L^1}\|D^s \na u\|_{L^2}
	+ \|D^s u\|_{L^2}\|V_{ij}^\eta\|_{L^1}\|\na u\|_{L^\infty}\Big)
	\|\na D^\alpha u\|_{L^2} \\
	&\le C^*\sum_{i,j=1}^nA_{ij}\|u\|_{H^s}\|\na u\|_{H^{s}}\|\na D^\alpha u\|_{L^2},
\end{align*}
where $C_M$ is the constant from Lemma \ref{lem.moser},
$C^*>0$ depends on $C_M$ and the constant of the embedding $H^s(\R^d)\hookrightarrow
L^\infty(\R^d)$, and we have used
$\|V_{ij}^\eta\|_{L^1}=A_{ij}$. Summation of all $|\alpha|\le s$ leads to
$$
  \frac12\frac{d}{dt}\|u\|_{H^s}^2 + \sigma\|\na u\|_{H^s}^2
	\le C^*\sum_{i,j=1}^n A_{ij}\|u\|_{H^s}\|\na u\|_{H^{s}}^2,
$$
which can be written as
\begin{equation}\label{2.ineq}
  \frac{d}{dt}\|u\|_{H^s}^2 
	+ 2\bigg(\sigma - C^*\sum_{i,j=1}^nA_{ij}\|u\|_{H^s}\bigg)
	\|\na u\|_{H^{s}}^2 \le 0.
\end{equation}
This inequality holds for all $t\in[0,T]$, where $T<T^*$.
By Lemma \ref{lem.sqrt}, applied to $f(t)=\|u(t)\|_{H^s}^2$, 
$g(t)=\|\na u(t)\|_{H^s}$, $a=\sigma$, and $b=C^*\sum_{i,j=1}^nA_{ij}$,
we find that $\|u(t)\|_{H^s}^2 \le (a/b)^2$ for $t\in[0,T]$. Here, we use
Assumption \eqref{1.small}. We deduce that $(d/dt)\|u\|_{H^s}^2\le 0$ and
consequently $\|u(t)\|_{H^s}\le \|u^0\|_{H^s}$ for $t\in[0,T]$.

Now, we take $u(T)$ as the initial datum for problem \eqref{1.uieta}. 
We deduce from Lemma \ref{lem.loc} 
the existence of a solution $u$ to \eqref{1.uieta} defined on $[T,T+T^*)$.
Here, $T^*>0$ can be chosen as the same end time as before since the norm 
of the initial datum $\|u(T)\|_{H^s}$ is not larger as $\|u^0\|_{H^s}$.
Note that $T^*$ becomes smaller only when the initial datum is larger in the
$H^s$ norm. Hence, $u(t)$
exists for $t\in[T,2T]$ and inequality \eqref{2.ineq} holds. As before, we conclude
from Lemma \ref{lem.sqrt} that $\|u(t)\|_{H^s}\le \|u^0\|_{H^s}$ for $t\in[T,2T]$. 
This argument can be continued, obtaining a global solution satisfying
$\|u(t)\|_{H^s}\le \|u^0\|_{H^s}$ for all $t>0$. 
Then, under the stronger assumption \eqref{1.small2},
$$
  \frac{d}{dt}\|u\|_{H^s}^2 \le
	-2\bigg(\sigma - C^*\sum_{i,j=1}^nA_{ij}\|u^0\|_{H^s}\bigg)\|\na u\|_{H^{s}}^2
	\le -\gamma\|\na u\|_{H^{s}}^2,
$$
which leads to \eqref{1.est}, finishing the proof.
\end{proof}

\begin{lemma}[Uniqueness of solutions]
Let assumption \eqref{1.small2} hold. 
Then the solution to problem \eqref{1.uieta} is unique
in the class of functions $u\in L^\infty(0,\infty;H^s(\R^d;\R^n))\cap
L^2(0,\infty;H^{s+1}(\R^d;$ $\R^n))$.
\end{lemma}

\begin{proof}
Let $u$ and $v$ be two solutions to \eqref{1.uieta} with the same initial data.
We multiply the difference of the equations satisfied by $u_i$ and $v_i$
by $u_i-v_i$, sum from $i=1,\ldots,n$, and integrate over $\R^d$.
Then, for all $0<t<T$ and some $T>0$,
\begin{align*}
  \frac12&\frac{d}{dt}\int_{\R^d}|u-v|^2 dx + \sigma\int_{\R^d}|\na(u-v)|^2 dx \\
	&\le \sum_{i,j=1}^n\big(\|u_i-v_i\|_{L^2}\|V_{ij}^\eta * u_j\|_{L^\infty}
	+ \|v_i\|_{L^\infty}\|V_{ij}^\eta * \na(u_j-v_j)\|_{L^2}\big)
	\|\na(u_i-v_i)\|_{L^2} \\
  &\le \sum_{i,j=1}^n\big(\|u-v\|_{L^2}\|V_{ij}^\eta\|_{L^1}\|\na u\|_{L^\infty}
	+ \|v\|_{L^\infty}\|V_{ij}^\eta\|_{L^1}\|\na(u-v)\|_{L^2}\big)\|\na(u-v)\|_{L^2} \\
	&\le \sum_{i,j=1}^nA_{ij}\big(\|u\|_{H^{s+1}}\|u-v\|_{L^2}\|\na(u-v)\|_{L^2}
	+ \|v\|_{H^s}\|\na(u-v)\|_{L^2}^2\big).
\end{align*}
By assumption, $\sum_{i,j=1}^nA_{ij}\|v\|_{H^s} \le \sigma-\gamma$
(since we supposed that $C^*\ge 1$). 
Thus, using the Cauchy-Schwarz inequality, it follows that
\begin{align*}
  \frac12&\frac{d}{dt}\int_{\R^d}|u-v|^2 dx + \sigma\int_{\R^d}|\na(u-v)|^2 dx \\
	&\le C(\eps)\bigg(\sum_{i,j=1}^n A_{ij}\|u\|_{H^{s+1}}\bigg)^2
	\|u-v\|_{L^2}^2 + \eps n^2\|\na(u-v)\|_{L^2}^2
	+ (\sigma-\gamma)\|\na(u-v)\|_{L^2}^2 \\
	&\le C(\eps)\|u(t)\|_{H^{s+1}}^2\|u-v\|_{L^2}^2 + \sigma\|\na(u-v)\|_{L^2}^2,
\end{align*}
if we choose $\eps\le \gamma/n^2$. Observe that the norm 
$\|u\|_{L^2(0,\infty;H^{s+1})}$ is bounded. This allows us to apply the Gronwall 
inequality, and together with the fact that $\|(u-v)(0)\|_{L^2}=0$, we infer that 
$\|(u-v)(t)\|_{L^2}=0$, concluding the proof.
\end{proof}


\section{Existence for the cross-diffusion system \eqref{1.cds}}\label{sec.ex2}

We prove Proposition \ref{prop.ex2} whose proof is split into two lemmas.

\begin{lemma}[Existence and uniqueness of solutions]
Let the assumptions of Proposition \ref{prop.ex2} hold. Then there exists
a unique solution to \eqref{1.cds} satisfying \eqref{1.estu}.
\end{lemma}

\begin{proof}
Let $u_\eta$ be the solution to \eqref{1.uieta}. We prove that a subsequence of
$(u_\eta)$ converges to a solution to problem \eqref{1.cds}. 
In view of the uniform estimate \eqref{1.est}, there exists a subsequence of
$(u_\eta)$, which is not relabeled, such that, as $\eta\to 0$,
\begin{equation}\label{3.u}
  u_\eta\rightharpoonup u\quad\mbox{weakly in }L^2(0,T;H^{s+1}(\R^d)).
\end{equation}
We show that $u$ is a weak solution to problem \eqref{1.cds}. First, we claim that
$$
  V_{ij}^\eta * \na u_{\eta,j} \rightharpoonup a_{ij}\na u_j
	\quad\mbox{weakly in }L^2(0,T;L^2(\R^d)).
$$
To prove this statement, we observe that $V_{ij}^\eta * \psi \to a_{ij}\psi$ 
strongly in $L^2(0,T;L^2(\R^d))$ for any $\phi\in L^2(0,T;L^2(\R^d))$ 
\cite[Theorem 9.10]{Rud87}
and $\na u_{\eta,j}\rightharpoonup \na u_j$ weakly in $L^2(0,T;L^2(\R^d))$. 
Therefore, for all $\psi\in C_0^\infty(\R^d\times(0,T);\R^n)$, as $\eta\to 0$,
\begin{align*}
  \bigg|\int_0^T&\int_{\R^d}\big(V_{ij}^\eta * \na u_{\eta,j} - a_{ij}\na u_j)
	\cdot\psi dxdt\bigg| \\
	&= \bigg|\int_0^T\int_{\R^d}\bigg(\int_{\R^d}V_{ij}^\eta(x-y)\na u_{\eta,j}(y,t)dy
	\bigg)\cdot\psi(x,t)dxdt \\
	&\phantom{xx}{}- \int_0^T\int_{\R^d}a_{ij}\na u_{j}(y,t)\cdot\psi(y,t)dydt\bigg| \\
	&\le \bigg|\int_0^T\int_{\R^d}\bigg(\int_{\R^d}V_{ij}^\eta(x-y)\psi(x,t)dx
	- a_{ij}\psi(y,t)\bigg)\cdot\na u_{\eta,j}(y,t)dydt\bigg| \\
	&\phantom{xx}{}+ |a_{ij}|\bigg|\int_0^T\int_{\R^d}(\na u_{\eta,j}-\na u_j)(y,t)\cdot
	\psi(y,t)dydt\bigg| \\
	&\le \|V_{ij}^\eta*\psi - a_{ij}\psi\|_{L^2(0,T;L^2(\R^d))}
	\|\na u_{\eta,j}\|_{L^2(0,T;L^2(\R^d))} \\
	&\phantom{xx}{}+ |a_{ij}|\bigg|\int_0^T\int_{\R^d}(\na u_{\eta,j}-\na u_j)(y,t)\cdot
	\psi(y,t)dydt\bigg| \to 0,
\end{align*}
which proves the claim. Estimate \eqref{1.est} and the embedding
$H^s(\R^d)\hookrightarrow L^\infty(\R^d)$ show that
\begin{equation}\label{3.aux}
  \big\|u_{\eta,i}V_{ij}^\eta * \na u_{\eta,j}\big\|_{L^2(0,T;L^2)} 
	\le \|u_{\eta,i}\|_{L^\infty(0,T;L^\infty)}\|V_{ij}^\eta\|_{L^1}
	\|\na u_{\eta,j}\|_{L^2(0,T;L^2)} \le C
\end{equation}
and consequently, for any $T>0$,
$$
  \|\pa_t u_{\eta,i}\|_{L^2(0,T;H^{-1})}
  \le \sigma_i\|\na u_{\eta,i}\|_{L^2(0,T;L^2)}
	+ \sum_{j=1}^n\big\|u_{\eta,i}V_{ij}^\eta * \na u_{\eta,j}\big\|_{L^2(0,T;L^2)} 
	\le C.
$$
The weak formulation of \eqref{1.uieta} reads as
\begin{equation}\label{3.weak}
  \int_0^T\langle\pa_t u_{\eta,i},\phi\rangle\zeta(t)dt
	= \int_0^T\int_{\R^d}\bigg(\sigma_i\na u_{\eta,i} + \sum_{j=1}^n u_{\eta,i}
	\na V_{ij}^\eta * \na u_{\eta,j}\bigg)\cdot\na\phi dx\zeta(t)dt,
\end{equation}
where $\phi\in C_0^\infty(\R^d)$ with $\operatorname{supp}(\phi)\subset B_R(0)$
and $\zeta\in C^\infty([0,T])$. Since the ball $B_R(0)$ is bounded and
the embedding $H^1(B_R(0))\hookrightarrow L^2(B_R(0))$ is compact, the
Aubin-Lions lemma \cite{Sim87} gives the existence of a subsequence of
$(u_\eta)$, which is not relabeled, such that $u_\eta\to u$ strongly in
$L^2(0,T;L^2(B_R(0)))$ as $\eta\to 0$, and the limit coincides with the
weak limit in \eqref{3.u}. We deduce that
$$
  u_{\eta,i}V_{ij}^\eta * \na u_{\eta,j} \rightharpoonup a_{ij}u_i\na u_j
	\quad\mbox{weakly in }L^1(0,T;L^1(B_R(0))).
$$
Estimate \eqref{3.aux} shows that this convergence even holds in 
$L^2(0,T;L^2(B_R(0)))$. We can perform the limit in \eqref{3.weak}, which shows
that the limit $u$ is a solution to the cross-diffusion problem \eqref{1.cds}. The
uniform estimates \eqref{1.estu} follow from \eqref{1.est} using the
lower semicontinuity of the norm. 

Next, we show the uniqueness of solutions.
Let $u$ and $v$ be two solutions to \eqref{1.cds} with the same initial data. 
Taking the difference of the
equations satisfied by $u_i$ and $v_i$, multiplying the resulting equation by
$u_i-v_i$, summing from $i=1,\ldots,n$, integrating over $\R^d$, and using the
Cauchy-Schwarz inequality leads to
\begin{align*}
  \frac12&\frac{d}{dt}\int_{\R^d}|u-v|^2 dx + \sigma\int_{\R^d}|\na(u-v)|^2 dx \\
  &\le \sum_{i,j=1}^n A_{ij}\|u\|_{H^{s+1}}\|u-v\|_{L^2}\|\na(u-v)\|_{L^2}
	+ \sum_{i,j}A_{ij}\|v\|_{H^s}\|u-v\|_{L^2}^2.
\end{align*}
In view of estimate \eqref{1.estu}, this becomes
$$
  \frac12\frac{d}{dt}\int_{\R^d}|u-v|^2 dx + \frac{\sigma}{2}\int_{\R^d}|\na(u-v)|^2 dx
  \le C(t)\|u-v\|_{L^2}^2,
$$
and the constant $C(t)>0$ is integrable (as it depends on $\|u(t)\|_{H^{s+1}}$).
Gronwall's inequality then implies that $(u-v)(t)=0$ for $t>0$.
\end{proof}

\begin{lemma}[Error estimate]
Let the assumptions of Proposition \ref{prop.ex2} hold. Let $u$ be
the solution to \eqref{1.cds} and $u_\eta$ be the solution to \eqref{1.uieta}.
Then the error estimate \eqref{1.error} holds.
\end{lemma}

\begin{proof}
We take the difference of equations \eqref{1.uieta} and \eqref{1.cds},
\begin{align*}
  \pa_t&(u_{\eta,i}-u_i) - \sigma_i\Delta(u_{\eta,i}-u_i)
	= \diver\bigg(\sum_{j=1}^n u_{\eta,i}V_{ij}^\eta * \na u_{\eta,j}
	- \sum_{i=1}^n a_{ij}u_i\na u_j\bigg) \\
	&= \diver\sum_{j=1}^n\Big((u_{\eta,i}-u_i)V_{ij}^\eta * \na u_{\eta,j}
	+ u_i(V_{ij}^\eta * \na u_{\eta,j}	- a_{ij}\na u_{\eta,j}) 
	+ a_{ij}u_i\na(u_{\eta,j}-u_j)\Big).
\end{align*}
Multiplying this equation by $u_{\eta,i}-u_i$, summing from $i=1,\ldots,n$, 
integrating over $\R^d$, using the Cauchy-Schwarz inequality, and the estimate
$|a_{ij}|\le A_{ij}$, we find that
\begin{align*}
  \frac12\frac{d}{dt}&\int_{\R^d}|u_\eta-u|^2 dx 
	+ \sigma\int_{\R^d}|\na(u_\eta-u)|^2 dx \\
	&\le \sum_{i,j=1}^n\|u_{\eta,i}-u_i\|_{L^2}\|V_{ij}^\eta 
	* \na u_{\eta,j}\|_{L^\infty}\|\na(u_{\eta,i}-u_i)\|_{L^2} \\
	&\phantom{xx}{}+ \sum_{i,j=1}^n \|u_i\|_{L^\infty}\|V_{ij}^\eta * \na u_{\eta,j}
	- a_{ij}\na u_{\eta,j}\|_{L^2}\|\na(u_{\eta,j}-u_j)\|_{L^2} \\
	&\phantom{xx}{}+ \sum_{i,j=1}^n A_{ij}\|u_i\|_{L^\infty}
	\|\na(u_{\eta,j}-u_j)\|_{L^2}\|\na(u_{\eta,i}-u_i)\|_{L^2} \\
	&= I_ 1 + I_2 + I_3.
\end{align*}

We estimate the right-hand side term by term. First, by the continuous embedding
$H^{s}(\R^d)\hookrightarrow L^\infty(\R^d)$,
\begin{align*}
  |I_1| &\le \sum_{i,j=1}^n\|u_{\eta,i}-u_i\|_{L^2}\|V_{ij}^\eta\|_{L^1}
	\|\na u_{\eta,j}\|_{L^\infty}\|\na(u_{\eta,i}-u_i)\|_{L^2} \\
  &\le C(\gamma)\bigg(\sum_{i,j=1}^n A_{ij}\|\na u_{\eta}\|_{H^{s}}\bigg)^2
	\|u_\eta-u\|_{L^2}^2 + \frac{\gamma}{4}\|\na(u_\eta-u)\|_{L^2}^2.
\end{align*}
To estimate $I_2$, let $g\in L^2(\R^d;\R^n)$. 
Since $\operatorname{supp}V_{ij}^\eta\subset B_\eta(0)$, the mean-value theorem
shows that
\begin{align*}
  \bigg|\int_{\R^d}&\big(V_{ij}^\eta * \na u_{\eta,j} - a_{ij}\na u_{\eta,j}\big)(x)
	\cdot g(x)dx\bigg| \\
	&= \bigg|\int_{\R^d}\int_{B_\eta(0)} V_{ij}^\eta(y)\sum_{k=1}^d
	\bigg(\frac{\pa u_{\eta,j}}{\pa x_k}(x-y)	- \frac{\pa u_{\eta,j}}{\pa x_k}(x)\bigg)
	g_k(x) dydx\bigg| \\
	&= \bigg|\int_{\R^d}\int_{B_\eta(0)} V_{ij}^\eta(y)\bigg(\int_0^1
	\sum_{k,\ell=1}^d\frac{\pa^2 u_{\eta,j}}{\pa x_k\pa x_\ell}(x-ry)y_\ell dr\bigg)g_k(x)
	dydx\bigg| \\
	&\le \eta\int_0^1\int_{B_\eta(0)} |V_{ij}^\eta(y)|\int_{\R^d}
	|D^2 u_{\eta,j}(x-ry)|\,|g(x)|dxdydr \\
	&\le \eta\int_0^1\int_{B_\eta(0)}|V_{ij}^\eta(y)|\,
	\|D^2 u_{\eta,j}(\cdot-ry)\|_{L^2}\|g\|_{L^2}dydr \\
	&\le \eta\|V_{ij}^\eta\|_{L^1}\|D^2 u_{\eta,j}\|_{L^2}\|g\|_{L^2}
	= \eta A_{ij}\|D^2 u_{\eta,j}\|_{L^2}\|g\|_{L^2}.
\end{align*}	
This shows that
\begin{equation}\label{3.diff}
  \|V_{ij}^\eta * \na u_{\eta,j} - a_{ij}\na u_{\eta,j}\|_{L^2} 
	\le \eta C\|D^2 u_{\eta,j}\|_{L^2} \le \eta C
\end{equation}
and consequently,
\begin{align*}
  |I_2| &\le C(\gamma)\sum_{i,j=1}^n\|V_{ij}^\eta * \na u_{\eta,j}
	- a_{ij}\na u_{\eta,j}\|_{L^2}^2
	+ \frac{\gamma}{4}\|\na(u_\eta-u)\|_{L^2}^2 \\
	&\le C(\gamma)\eta^2 + \frac{\gamma}{4}\|\na(u_\eta-u)\|_{L^2}^2.
\end{align*}
Finally, by Assumption \eqref{1.small2},
$$
  |I_3| \le \bigg(\sum_{i,j=1}^nA_{ij}\|u\|_{H^{s+1}}\bigg)
	\|\na(u_\eta-u)\|_{L^2}^2 
	\le (\sigma-\gamma)\|\na(u_\eta-u)\|_{L^2}^2.
$$
Therefore,
\begin{align*}
  \frac12\frac{d}{dt}&\int_{\R^d}|u_\eta-u|^2 dx 
	+ \frac{\gamma}{2}\int_{\R^d}|\na(u_\eta-u)|^2 dx \\
	&\le C(\gamma)\bigg(\sum_{i,j=1}^nA_{ij}\|\na u_{\eta}\|_{H^{s}}\bigg)^2
	\|u_\eta-u\|_{L^2}^2 + C(\gamma)\eta^2 \\
	&\le C(\gamma)(\sigma-\gamma)^2\|u_\eta-u\|_{L^2}^2 + C(\gamma)\eta^2,
\end{align*}
and Gronwall's lemma gives the conclusion.
\end{proof}


\section{Existence of solutions to the stochastic systems}\label{sec.stoch}

We prove the solvability of the stochastic ordinary differential systems
\eqref{1.Xi}, \eqref{1.barXi}, and \eqref{1.hatXi}. 

\begin{lemma}[Solvability of the stochastic many-particle system]\label{lem.many}
For any fixed $\eta>0$, problem \eqref{1.Xi} has a pathwise unique strong solution
$X_{\eta,i}^{k,N}$ that is $\F_t$-adapted.
\end{lemma}

\begin{proof}
By assumption, the gradient $\na V_{ij}^\eta$ is bounded and Lipschitz
continuous. Then \cite[Theorem 5.2.1]{Oek00} or \cite[Theorem 3.1.1]{PrRo07} show that
there exists a (up to $\mathbb{P}$-indistinguishability) 
pathwise unique strong solution to \eqref{1.Xi}. 
\end{proof}

\begin{lemma}[Solvability of the nonlocal stochastic system]\label{lem.nonloc}
Let $u_\eta$ be a solution to the nonlocal diffusion system \eqref{1.uieta}
satisfying $|\na u_\eta|\in L^\infty(0,\infty;W^{1,\infty}(\R^d;\R^n))$.
Then problem \eqref{1.barXi} has a pathwise unique strong solution $\bar X_{\eta,i}^k$
with probability density function $u_{\eta,i}$.
\end{lemma}

\begin{remark}\label{rem}\rm
We have shown in Proposition \ref{prop.ex2} that if $u^0\in H^s(\R^d;\R^n)$ and the 
smallness condition \eqref{1.small} holds, there exists a unique
solution $u_\eta\in L^\infty(0,\infty;H^s(\R^d;\R^n))$. 
The regularity for $u_\eta$, required in Lemma \ref{lem.nonloc}, is fulfilled 
if $s>d/2+2$, thanks to the embedding $u_{\eta,i}\in L^\infty(0,\infty;H^s(\R^d))
\hookrightarrow L^\infty(0,\infty;W^{2,\infty}(\R^d))$ for $i=1,\ldots,n$.
\qed
\end{remark}

\begin{proof}[Proof of Lemma \ref{lem.nonloc}]
We proceed as in the proof of Lemma 3.2 of \cite{CGK18}.
Let $v$ be a solution to \eqref{1.uieta} satisfying $v(0)=u^0$ in $\R^d$,
where $u^0_i$ is the density of $\xi^k_i$.
By assumption, $\na V_{ij}^\eta * v_j = V_{ij}^\eta * \na v_j$
is bounded and Lipschitz continuous. Therefore, 
$$
  d\bar X_{\eta,i}^k = -\sum_{j=1}^n(\na V_{ij}^\eta * v_j)(\bar X_{\eta,i}(t),t)dt
	+ \sqrt{2\sigma_i}dW_i^k(t), \quad \bar X_{\eta,i}^k(0)=\xi_i^k,
$$
has a pathwise unique strong solution $\bar X_{\eta,i}^k$. Let $u_{\eta,i}$
be the probability density function of $\bar X_{\eta,i}^k$ and let $\phi_i$ be
a smooth test function. Then It\^o's lemma implies that
\begin{align*}
  \phi_i(&\bar X_{\eta,i}^k(t),t) - \phi_i(\xi_i^k,0) \\
	&= \int_0^t\pa_s\phi_i(\bar X_{\eta,i}^k(s),s)ds
	- \sum_{j=1}^n\int_0^t(V_{ij}^\eta * \na v_j)(\bar X_{\eta,i}^k(s),s)
	\cdot\na \phi_i(\bar X_{\eta,i}^k(s),s)ds \\
	&\phantom{xx}{}+ \sigma_i\int_0^t\Delta\phi_i(\bar X_{\eta,i}^k(s),s)ds 
	+ \sqrt{2\sigma_i}\int_0^t\na\phi_i(\bar X_{\eta,i}^k(s),s)dW_i^k(s).
\end{align*}
Applying the expectation
$$
  \E\big(\phi_i(\bar X_{\eta,i}^k(t),t)\big) 
	= \int_{\R^d}\phi_i(x,t)u_{\eta,i}(x,t)dx
$$
to the previous expression yields
\begin{align*}
  \int_{\R^d}\phi_i(x,t)u_{\eta,i}(x,t)dx
	&= \int_{\R^d}\phi_i(x,0)u_{\eta,i}(x,0)dx \\
	&\phantom{xx}{}
	+ \int_0^t\int_{\R^d}\big(\pa_s\phi_i(x,s) + \sigma_i\Delta\phi_i(x,s)\big)
	u_{i,\eta}(x,s)dxds \\
	&\phantom{xx}{}
	- \sum_{j=1}^n\int_0^t\int_{\R^d}\na\phi_i(x,s)\cdot(V_{ij}^\eta * \na v_j)(x,s)
	u_{\eta,i}(x,s)dxds.
\end{align*}
This is the weak formulation of 
$$
  \pa_t u_{\eta,i} - \sigma_i\Delta u_{\eta,i}
	= \diver\bigg(\sum_{j=1}^n u_{\eta,i}V_{ij}^\eta * \na v_j\bigg),
	\quad u_{\eta,i}(0)=u^0.
$$
The unique solvability of problem \eqref{1.uieta} implies that the
solution is $u_{\eta,i}$, and we obtain $v=u_\eta$. This finishes the proof.
\end{proof}

By the same technique, the solvability of the limiting stochastic system
can be proved.

\begin{lemma}[Solvability of the limiting stochastic system]\label{lem.limit}
Let $u$ be the unique solution to problem \eqref{1.cds} satisfying
$|\na u|\in L^\infty(0,\infty;W^{1,\infty}(\R^d;\R^n))$. Then there exists
a pathwise unique strong solution $\widehat X_i^k$ with probability density
function $u_i$.
\end{lemma}


\section{Proof of Theorem \ref{thm.main}}\label{sec.proof}

First, we show an estimate for the difference $X_{\eta,i}^{k,N}-\bar X_{\eta,i}^k$.

\begin{lemma}\label{lem.error1}
Let the assumptions of Theorem \ref{thm.main} hold. Then, for any $t>0$,
$$
  \E\bigg(\sum_{i=1}^n\sup_{0<s<t}\sup_{k=1,\ldots,N}
	\big|X_{\eta,i}^{k,N}(s) - \bar X_{\eta,i}^k(s)\big|^2\bigg) 
	\le \frac{C(t)}{N^{1-C(t)\eps}},
$$
where the constant $C(t)$ depends on $t$, $n$, $\|D^2 V_{ij}\|_{L^\infty}$,
and the initial datum $u^0$.
\end{lemma}

\begin{proof}
We set
$$
  S_t = \sum_{i=1}^n\sup_{0<s<t}\sup_{k=1,\ldots,N}
	\big|X_{\eta,i}^{k,N}(s) - \bar X_{\eta,i}^k(s)\big|^2.
$$
The difference of equations \eqref{1.Xi} and \eqref{1.barXi}, satisfied by
$X_{\eta,i}^{k,N}$ and $\bar X_{\eta,i}^k$, respectively, equals
\begin{align*}
  &X_{\eta,i}^{k,N}(t) - \bar X_{\eta,i}^k(t) \\
	&\phantom{x}{}= -\int_0^t \sum_{j=1}^n\frac{1}{N}\sum_{\ell=1}^N\Big(
	\na V_{ij}^\eta\big(X_{\eta,i}^{k,N}(s) - X_{\eta,j}^{\ell,N}(s)\big)
	- (\na V_{ij}^\eta * u_{\eta,j})(\bar X_{\eta,i}^k(s),s)\Big)ds
\end{align*}
and thus,
\begin{align*}
  \sum_{i=1}^n&\sup_{k=1,\ldots,N}
	\big|X_{\eta,i}^{k,N}(s) - \bar X_{\eta,i}^\ell(s)\big|^2
	\le \sum_{i=1}^n\bigg(\int_0^t\frac{1}{N}\sum_{j=1}^n\sup_{k=1,\ldots,N} \\
	&{}\times\bigg|\sum_{\ell=1}^N\Big(
	\na V_{ij}^\eta\big(X_{\eta,i}^{k,N}(s)-X_{\eta,j}^{\ell,N}(s)\big)
	- (\na V_{ij}^\eta * u_{\eta,j})(\bar X_{\eta,i}^k(s),s)\Big)\bigg|ds\bigg)^2.
\end{align*}
Taking the supremum in $(0,t)$ and the expectation and using the
Cauchy-Schwarz inequality with respect to $t$ yields
\begin{align*}
  \E(S_t) &\le \sum_{i=1}^n\frac{t}{N^2}
	\int_0^t\E\bigg(\sum_{j=1}^n\sup_{k=1,\ldots,N} \\
	&\phantom{xx}{}{}\times\bigg|\sum_{\ell=1}^N\Big(
	\na V_{ij}^\eta\big(X_{\eta,i}^{k,N}(s)-X_{\eta,j}^{\ell,N}(s)\big)
	- (\na V_{ij}^\eta * u_{\eta,j})(\bar X_{\eta,i}^k(s),s)\Big)\bigg|ds\bigg)^2 \\
	&\le \sum_{i=1}^n\frac{t}{N^2}\int_0^t\bigg\{
	\E\sum_{j=1}^n\sup_{k=1,\ldots,N} \\
	&\phantom{xx}{}\times\bigg|\sum_{\ell=1}^N\Big(
	\na V_{ij}^\eta\big(X_{\eta,i}^{k,N}(s)-X_{\eta,j}^{\ell,N}(s)\big)
	- \na V_{ij}^\eta\big(X_{\eta,i}^{k,N}(s)-\bar X_{\eta,j}^{\ell}(s)\big)
	\Big)\bigg|^2 \\
	&\phantom{xx}{}+ \E\sum_{j=1}^n\sup_{k=1,\ldots,N}\bigg|\sum_{\ell=1}^N
	\Big(\na V_{ij}^\eta\big(X_{\eta,i}^{k,N}(s)-\bar X_{\eta,j}^{\ell}(s)\big)
	- \na V_{ij}^\eta\big(\bar X_{\eta,i}^{k}(s)-\bar X_{\eta,j}^{\ell}(s)\big)\Big)
	\bigg|^2 \\
	&\phantom{xx}{}+ \E\sum_{j=1}^n\sup_{k=1,\ldots,N}\bigg|\sum_{\ell=1}^N
	\Big(\na V_{ij}^\eta\big(\bar X_{\eta,i}^{k}(s)-\bar X_{\eta,j}^{\ell}(s)\big)
	- (\na V_{ij}^\eta * u_{\eta,j})(\bar X_{\eta,i}^k(s),s)\Big)
	\bigg|^2\bigg\}ds \\
	&=: J_1 + J_2 + J_3.
\end{align*}
We estimate the terms $J_1$, $J_2$, and $J_3$ separately. 

Let $L_{ij}^\eta$ be the Lipschitz constant of $\na V_{ij}^\eta$. Because of
$V_{ij}^\eta(z) = \eta^{-d}V_{ij}(|z|/\eta)$, we obtain
$L_{ij}^\eta = \eta^{-d-2}\|D^2 V_{ij}\|_{L^\infty}$. Hence,
\begin{align*}
  |J_1| &\le \sum_{i=1}^n\frac{t}{N^2}\int_0^t\E\bigg(\sum_{j=1}^n(L_{ij}^\eta)^2
	\bigg(\sum_{\ell=1}^N\big|X_{\eta,j}^{\ell,N}(s)-\bar X_{\eta,j}^\ell(s)\big|\bigg)^2
	\bigg)ds \\
	&\le t\sum_{i,j=1}^n(L_{ij}^\eta)^2\int_0^t\E\bigg(\sup_{\ell=1,\ldots,N}
	\big|X_{\eta,j}^{\ell,N}(s)-\bar X_{\eta,j}^\ell(s)\big|^2\bigg)ds \\
	&\le \frac{tn}{\eta^{2d+4}}\sup_{i,j=1,\ldots,n}\|D^2 V_{ij}\|_{L^\infty}^2
	\int_0^t\E(S_s)ds.
\end{align*}
Furthermore, by similar arguments,
\begin{align*}
  |J_2| &\le \sum_{i=1}^n\frac{t}{N^2}\int_0^t\E\bigg(\sum_{j=1}^n(L_{ij}^\eta)^2
	\sup_{k=1,\ldots,N}\bigg(\sum_{\ell=1}^N
	\big|X_{\eta,i}^{k,N}(s)-\bar X_{\eta,j}^k(s)\big|\bigg)^2\bigg)ds \\
	&\le t\sum_{i,j=1}^n(L_{ij}^\eta)^2\int_0^t\E\bigg(\sup_{k=1,\ldots,N}
	\big|X_{\eta,i}^{k,N}(s)-\bar X_{\eta,j}^k(s)\big|^2\bigg)ds \\
  &\le \frac{tn}{\eta^{2d+4}}\sup_{i,j=1,\ldots,n}\|D^2V_{ij}\|_{L^\infty}^2
	\int_0^t\E(S_s)ds.
\end{align*}
For the third term, we set
$$
  Z_{i,j}^{k,\ell}(s) := \na V_{ij}^\eta
	\big(\bar X_{\eta,i}^{k}(s)-\bar X_{\eta,j}^{\ell}(s)\big)
	- (\na V_{ij}^\eta * u_{\eta,j})(\bar X_{\eta,i}^k(s),s),
$$
write the square as a product of two sums, and use the independence of
$Z_{i,j}^{k,1},\ldots,Z_{i,j}^{k,N}$:
\begin{align*}
  |J_3| &= \sum_{i=1}^n\frac{t}{N^2}\int_0^t
	\E\sum_{j=1}^n\sup_{k=1,\ldots,N}\bigg|\sum_{\ell=1}^N 
	Z_{i,j}^{k,\ell}(s)\bigg|^2 ds \\
	&= \sum_{i,j=1}^n\frac{t}{N^2}\int_0^t\E\bigg(\sup_{k=1,\ldots,N}
	\sum_{\ell=1}^N Z_{i,j}^{k,\ell}(s)\cdot\sum_{m=1}^N Z_{i,j}^{k,m}(s)\bigg) ds \\
	&\le \sum_{i,j=1}^n\frac{t}{N^2}\sum_{\ell,m=1}^N\int_0^t\sup_{k=1,\ldots,N}
	\E\bigg(Z_{i,j}^{k,\ell}(s)\cdot Z_{i,j}^{k,m}(s)\bigg)ds \\
	&= \sum_{i,j=1}^n\frac{t}{N^2}\sum_{\ell=1}^N\int_0^t\sup_{k=1,\ldots,N}
	\E\big|Z_{i,j}^{k,\ell}(s)\big|^2 ds \\
	&\phantom{xx}{}
	+ \sum_{i,j=1}^n\frac{t}{N^2}\sum_{\ell\neq m}\int_0^t\sup_{k=1,\ldots,N}
	\E\big(Z_{i,j}^{k,\ell}(s)\big)\E\big(Z_{i,j}^{k,m}(s)\big)ds.
\end{align*}

We claim that the expectation  of $Z_{i,j}^{k,\ell}$ vanishes. Indeed, since
$\bar X^k_{\eta,i}$ and $\bar X^\ell_{\eta,j}$ are independent with
distribution functions $u_{\eta,i}$ and $u_{\eta,j}$, the joint distribution
is $u_{\eta,i}\otimes u_{\eta,j}$. This gives
\begin{align*}
  \E\big(Z_{ij}^{k,\ell}(s)\big)
	&= \int_{\R^d}\int_{\R^d}\big(\na V_{ij}^\eta(\xi-x) 
	- (\na V_{ij}^\eta * u_{\eta,j})(\xi,s)\big)u_{\eta,i}(\xi,s)u_{\eta,j}(x,s)dxd\xi \\
	&= \int_{\R^d} u_{\eta,i}(\xi,s)\bigg(\int_{\R^d} \na V_{ij}^\eta(\xi-x)u_{\eta,j}(x,s)
	dx - (\na V_{ij}^\eta * u_{\eta,j})(\xi,s)\bigg) d\xi = 0.
\end{align*}
Therefore, using the estimates $|\na V_{ij}^\eta|\le C\eta^{-d-1}$,
$\|\na V_{ij}^\eta\|_{L^2}\le C\eta^{-1}$, and consequently
$$
  \|\na V_{ij}^\eta * u_{\eta,j}\|_{L^\infty}
	\le \|\na V_{ij}^\eta\|_{L^2}\|u_{\eta,j}\|_{L^2}
	\le C\eta^{-1}, 
$$
we deduce that
$$
  |J_3| = \sum_{i,j=1}^n\frac{t}{N^2}\sum_{\ell=1}^N\int_0^t\sup_{k=1,\ldots,N}
	\E\big|Z_{i,j}^{k,\ell}(s)\big|^2 ds
  \le \frac{n^2}{N}\frac{Ct^2}{\eta^{2d+2}}.
$$

Summarizing these estimations, we conclude that
$$ 
  \E(S_t) \le \frac{tn}{\eta^{2d+4}}\sup_{i,j=1,\ldots,n}\|D^2 V_{ij}\|_{L^\infty}^2
	\int_0^t\E(S_s)ds + \frac{n^2}{N}\frac{Ct^2}{\eta^{2d+2}},
$$
and, by Gronwall's inequality,
$$
  \E(S_t) \le \frac{Ct^2}{N\eta^{2d+2}}\exp\bigg(\frac{Ct}{\eta^{2d+4}}\bigg),
	\quad t>0.
$$
For fixed $\eps\in(0,1)$ and $\eta\in(0,1)$, we choose $N\in\N$ such that
$\eps\log N\ge \eta^{-2d-4}$. Using $\eta^{-2d-2}\le\eta^{-2d-4}
\le \exp(C\eta^{-2d-4})$ for $C\ge 1$, we obtain
\begin{align*}
  \E(S_t) &\le \frac{C}{N\eta^{2d+2}}\exp\bigg(\frac{Ct}{\eta^{2d+4}}\bigg)
	\le \frac{C}{N}\exp\bigg(\frac{C(1+t)}{\eta^{2d+4}}\bigg) \\
	&\le \frac{C}{N}\exp\big(C\eps(1+t)\log N\big)
	\le CN^{-1+C\eps(1+t)}.
\end{align*}
This proves the result.
\end{proof}

Next, we prove an estimate for the difference $\bar X_{\eta,i}^k-\widehat X_i^k$.

\begin{lemma}\label{lem.error2}
Let the assumptions of Theorem \ref{thm.main} hold and let $s>d/2+2$, $t>0$. Then
$$
  \E\bigg(\sum_{i=1}^n\sup_{0<s<t}\sup_{k=1,\ldots,N}\big|\bar X_{\eta,i}^k(s)
	- \widehat X_i^k(s)\big|\bigg) \le C(t)\eta.
$$
\end{lemma}

\begin{proof}
We use similar arguments as in the proof of Lemma \ref{lem.error1}. 
Taking the difference of equations \eqref{1.barXi} and \eqref{1.hatXi}, satisfied by 
$\bar X_{\eta,i}^k$ and $\widehat X_i^k$, respectively, and setting
$$
  U_t = \sum_{i=1}^n\sup_{0<s<t}\sup_{k=1,\ldots,N}
	\big|\bar X_{\eta,i}^k(s)-\widehat X_i^k(s)\big|,
$$
if follows that
\begin{align*}
  \E(U_t) &\le \sum_{i=1}^n\int_0^t\E\bigg(\sum_{j=1}^n\sup_{k=1,\ldots,N}
	\Big|a_{ij}\na u_j(\widehat X_i^k(s),s) - (V_{ij}^\eta * \na u_{\eta,j})
	(\bar X_{\eta,i}^k(s),s)\Big|\bigg)ds \\
  &\le \sum_{i=1}^n\int_0^t\E\bigg(\sum_{j=1}^n\sup_{k=1,\ldots,N}
	\Big|a_{ij}\na u_j(\widehat X_i^k(s),s) - (V_{ij}^\eta * \na u_{j})
	(\widehat X_{i}^k(s),s)\Big|\bigg)ds \\
  &\phantom{xx}{}+ \sum_{i=1}^n\int_0^t\E\bigg(\sum_{j=1}^n\sup_{k=1,\ldots,N}
	\Big|(V_{ij}^\eta * \na u_{j})(\widehat X_{\eta,i}^k(s),s)
  - (V_{ij}^\eta * \na u_{\eta,j})(\widehat X_{i}^k(s),s)\Big|\bigg)ds \\
	&\phantom{xx}{}+ \sum_{i=1}^n\int_0^t\E\bigg(\sum_{j=1}^n\sup_{k=1,\ldots,N}
	\Big|(V_{ij}^\eta * \na u_{\eta,j})(\widehat X_{i}^k(s),s)
  - (V_{ij}^\eta * \na u_{\eta,j})(\bar X_{\eta,i}^k(s),s)\Big|\bigg)ds \\
	&=: K_1 + K_2 + K_3.
\end{align*}
Using \eqref{3.diff}, the first two terms on the right-hand side are 
estimated according to
\begin{align*}
  K_1 &\le \sum_{i,j=1}^n\int_0^t\int_{\R^d}\big|\big(a_{ij}\na u_j(x,s)
	- (V_{ij}^\eta * \na u_j)(x,s)\big)u_i(x,s)\big|dxds \\
	&\le n^2\max_{i,j=1,\ldots,n}\big\|a_{ij}\na u_j-V_{ij}^\eta * \na u_j
	\big\|_{L^2(0,t;L^2)}\|u_i\|_{L^2(0,t;L^2)} \\
	&\le C(n)\eta\|D^2 u\|_{L^2(0,t;L^2)} \le C\eta, \\
	K_2 &\le \sum_{i,j=1}^n\int_0^t\int_{\R^d}\big|V_{ij}^\eta * \na(u_j-u_{\eta,j})
	u_i\big| dxds \\
	&\le n^2\max_{i,j=1,\ldots,n}
	\|V_{ij}^\eta\|_{L^1}\|\na(u_j-u_{\eta,j})\|_{L^2(0,t;L^2)}
	\|u_i\|_{L^2(0,t;L^2)} \le C\eta,
\end{align*}
where we have used Lemma \ref{lem.young} (ii) and the error estimate from 
Lemma \ref{lem.error1}. Finally, the term $K_3$ can be controlled by
\begin{align*}
  K_3 &\le \sum_{i,j=1}^n \|\na(V_{ij,}^\eta * \na u_{\eta,j})\|_{L^\infty}
	\int_0^t\E(U_s)ds \\
	&\le n^2\max_{i,j=1,\ldots,n}\|V_{ij}^\eta\|_{L^1}\|D^2 u_\eta\|_{L^\infty}
	\int_0^t\E(U_s)ds \le C\int_0^t\E(U_s)ds.
\end{align*}
We need the assumption $s>d/2+2$ for the continuous embedding
$H^s(\R^d)\hookrightarrow W^{2,\infty}(\R^d)$, which allows us to estimate
$D^2 u_\eta$ in $L^\infty(\R^d)$. This shows that
$$
  \E(U_t) \le C\eta + C\int_0^t\E(U_s)ds,
$$
and Gronwall's inequality yields $\E(U_t)\le C(t)\eta$ for $t>0$.
The statement of the lemma follows after taking the supremum over $t>0$.
\end{proof}

Lemmas \ref{lem.error1} and \ref{lem.error2} imply Theorem \ref{thm.main}. Indeed,
it follows that
\begin{align*}
  \E&\bigg(\sum_{i=1}^n\sup_{0<s<t}\sup_{k=1,\ldots,N}
	\big|X_{\eta,i}^{k,N}(s) - \widehat X_i^k(s)\big|\bigg) \\
	&\le \E\bigg(\sum_{i=1}^n\sup_{0<s<t}\sup_{k=1,\ldots,N}
	\big|X_{\eta,i}^{k,N}(s) - \bar X_{\eta,i}^k(s)\big|\bigg)
	+ \E\bigg(\sum_{i=1}^n\sup_{0<s<t}\sup_{k=1,\ldots,N}
	\big|\bar X_{\eta,i}^k(s) - \widehat X_i^k(s)\big|\bigg) \\
  &\le CN^{(-1+C(t)\eps)/2} + C(t)\eta \le C(t)\eta,
\end{align*}
since the choice $\log N\ge 1/(\eps\eta^{2d+4})$ is equivalent to
$$
  N^{(-1+C(t)\eps)/2} \le \exp\bigg(\frac{1}{2\eps}(-1+C(t)\eps)\eta^{-2d-4}\bigg),
$$
and the right-hand side is smaller than $\eta$ possibly times a constant $C(t)$.


\begin{appendix}
\section{Some auxiliary results}

We recall some auxiliary results. 

\begin{lemma}[Moser-type calculus inequality; Prop.~2.1 in \cite{Maj84}]
\label{lem.moser}
Let $f$, $g\in H^s(\R^d)$ with $s>d/2+1$ and let 
$\alpha \in \N^d$ be a multi-index of order $|\alpha|\le s$. 
Then, for some constant $C_M>0$ only depending on $s$ and $d$,
$$
  \|D^\alpha(fg)\|_{L^2} \le C_M\big(\|f\|_{L^\infty}\|D^s g\|_{L^2}
	+ \|g\|_{L^\infty}\|D^s f\|_{L^2}\big).
$$
\end{lemma}

\begin{lemma}[Young's convolution inequality; formula (7) on page 107 
in \cite{LiLo01}]\label{lem.young}
{\rm (i)} Let $1\le p, q\le \infty$ satisfying $1/p+1/q=1/r+1$ and
$f\in L^p(\R^d)$, $g\in L^q(\R^d)$. Then
$$
  \|f*g\|_{L^r} \le \|f\|_{L^p}\|g\|_{L^q}.
$$

{\rm (ii)} Let $p$, $q$, $r\ge 1$ satisfying $1/p+1/q+1/r=2$ and
$f\in L^p(\R^d)$, $g\in L^q(\R^d)$, $h\in L^r(\R^d)$. Then, for some constant
$C>0$ only depending on $p$, $q$, $r$, and $d$,
$$
  \bigg|\int_{\R^d}(f * g)h dx\bigg|
	\le C\|f\|_{L^p}\|g\|_{L^q}\|h\|_{L^r}. 
$$
\end{lemma}

\begin{lemma}[Gronwall-type inequality]\label{lem.sqrt}
Let $a$, $b>0$, $g\in C^0([0,T])$ with $g(t)\ge 0$ for $t\in[0,T]$ and 
$f:[0,T]\to[0,\infty)$ be absolutely continuous such that
$$
  f'(t) \le -g(t)\big(a-b\sqrt{f(t)}\big)\quad\mbox{for }t>0
$$
and $0<f(0) \le (a/b)^2$. Then $f(t) \le (a/b)^2$ for all $t\in[0,T]$.
\end{lemma}

We present a proof of this lemma since we could not find a reference in the
literature.

\begin{proof}
First, let $f(0) < (a/b)^2$. We claim that $f(t)<(a/b)^2$ for $t\in[0,T]$.
Assume that there exists $t_0\in[0,T]$ such that $f(t_0)\ge (a/b)^2$. By continuity,
there exists $t_1\in[0,t_0]$ such that $f(t_1)<(a/b)^2$ and $f'(t_1) > 0$. This
leads to the contradiction $0 < f'(t_1) \le -g(t_1)(a-b\sqrt{f(t_1)}) \le 0$, 
proving the claim.

Since $f(t)<(a/b)^2$ for $t\in[0,T]$, the differential inequality can be
written as
\begin{equation}\label{a.aux}
  \frac{f'(t)}{a-b\sqrt{f(t)}} \le -g(t), \quad t\in[0,T].
\end{equation}
Introduce  
$$
  F(s) = \int_0^s\frac{d\sigma}{a-b\sqrt{\sigma}}
	= \int_0^{\sqrt{s}}\frac{2\tau d\tau}{a-b\tau}
	= -\frac{2}{b^2}\big(b\sqrt{s} + a\log(a-b\sqrt{s}) - a\ln a\big).
$$
Then integrating \eqref{a.aux} over $(0,t)$ leads to
$$
  F(f(t)) - F(f(0)) \le -\int_0^t g(\tau)d\tau,
$$
which, after a computation, is equivalent to 
$$
  a - b\sqrt{f(t)} \ge \big(a-b\sqrt{f(0)}\big)\exp\bigg(\frac{b}{a}\sqrt{f(0)} 
	- \frac{b}{a}\sqrt{f(t)}
	+ \frac{b^2}{2a}\int_0^t g(\tau)d\tau\bigg).
$$
Finally, we choose a sequence of initial data $f_0^\delta < (a/b)^2$
such that $f_0^\delta\to f_0\le (a/b)^2$ as $\delta\to 0$. To each $f_0^\delta$,
we associate a function $f^\delta$ satisfying the differential inequality.
The proof shows that $f^\delta(t)<(a/b)^2$. In the limit $\delta\to 0$, this
reduces to $f(t)\le (a/b)^2$, where $f(t)=\lim_{\delta\to 0}f^\delta(t)$
for $t\in[0,T]$. 
\end{proof}

\end{appendix}


\end{document}